\documentclass[11pt]{amsart} 

\usepackage{amssymb, amsfonts, amsthm,amsmath}
\usepackage{a4wide}

\newcommand{\Bin}[1]{\operatorname{Bin}\left(#1\right)}

\newcommand{\PPo}[1]{\PP \left[\,#1\,\right]}

\newcommand{\Ex}[1]{\mathbb{E} \left[\, #1\,\right]}

\newcommand{\Var}[1]{\mathrm{Var}\left[#1\right]}
\newcommand{\Cov}[2]{\mathrm{Cov} \left(#1,#2 \right)}

\newcommand{\EE}{\mathbb{E}}
\newcommand{\PP}{\mathbb{P}}
\newcommand{\Ev}[1]{\mathcal{E}_{#1}}

\newcommand{\rt}[1]{\bar{r}_{#1}}
\newcommand{\st}[1]{\bar{s}_{#1}}

\newcommand{\eps}{\varepsilon}

\newcommand{\Gnp}{G(n,p)}
\newcommand{\VertexSet}{V_n}
\newcommand{\EdgeProb}{p}

\newcommand{\mm}[1]{\mu_{++} (#1)}

\newcommand{\sgn}{\mathrm{sgn}}

\newtheorem{theorem}{Theorem}[section]

\newtheorem{lemma}[theorem]{Lemma}

\newtheorem{claim}[theorem]{Claim}

\newtheorem{conjecture}[theorem]{Conjecture}

\usepackage{xcolor}

\begin{document}

\title[Rapid stabilisation in dense random graphs]{Resolution of a conjecture on majority dynamics: rapid stabilisation in dense random graphs}

\author{Nikolaos Fountoulakis\textsuperscript{1}}
\thanks{\textsuperscript{1}Part of this research was carried out while this author was visiting TU Graz, supported by TU Graz. Research partially supported by the Alan Turing Institute 
grant EP/N5510129/1.}
\address{School of Mathematics, University of Birmingham, B15 2TT, Birmingham, United Kingdom. {\tt {n.fountoulakis@bham.ac.uk}}}
\author{Mihyun Kang\textsuperscript{2}}
\address{Institute of Discrete Mathematics, Graz University of Technology, Steyrergasse 30, 8010, Graz, Austria.  {\tt {kang@math.tugraz.at}}}
\thanks{\textsuperscript{2}supported by  Austrian Science Fund (FWF): I3747 and  W1230III}
\author{Tam\'as Makai\textsuperscript{3}}
\address{Department of Mathematics G. Peano, University of Torino, Via Carlo Alberto 10, 10123, Torino, Italy. {\tt {tamas.makai@unito.it}}}
\address{School of Mathematics and Statistics, The University of New South Wales, Sydney, NSW 2052, Australia. {\tt {t.makai@unsw.edu.au}}}
\thanks{\textsuperscript{3}supported by  	
EPSRC grant EP/S016694/1, ``Memory in Evolving Graphs" (Compagnia di San Paolo/Universit\`a degli Studi di Torino) and ARC grant DP190100977}

\begin{abstract}
We study majority dynamics on the binomial random graph $G(n,p)$ 
with $p = d/n$ and $d > \lambda n^{1/2}$, for some large $\lambda>0$. 
In this process, each vertex has a state in $\{-1,+1 \}$ and at each round every vertex adopts the state of the majority of its neighbours, retaining its state in the case of a tie. 

We show that with high probability the process reaches unanimity in  at most four rounds. This confirms a 
conjecture of Benjamini, Chan, O'Donnel, Tamuz and Tan. 
\end{abstract}

\maketitle

\section{Introduction} 
Majority dynamics is a process on a graph $G=(V,E)$ which evolves in discrete steps and  
at step $t\geq 0$, every vertex $v \in V$ has state $S_t(v) \in \{ -1,+1\}$. The state of 
each vertex changes according to the majority of its neighbours in $G$. Namely, 
given the configuration $\{ S_t (v)\}_{v \in V}$ just 
after step $t$, vertex $v\in V$ has state 
$$S_{t+1}(v) :=
\begin{cases}
 \sgn \left(\sum_{u\in N(v)} S_t(u)\right) & \quad \text{if}\quad \sum_{u\in N(v)} S_t(u) \not = 0,\\
 S_t (v) & \quad \text{otherwise},
\end{cases}
$$
where $N(v)$ denotes the set of vertices adjacent to $v$ in $G$, and $\sgn(x)=-1$ if $x<0$ and $+1$ if $x>0$.  
In other words, $v$ adopts the majority of its neighbours, whereas, in the case of a tie, it retains its current state. 

This class of processes can be seen as a generalisation of a cellular automata such as those introduced 
by von Neumann~\cite{bk:vonNeumann}. In particular, it can be seen as a variation of the well-known Conway's \emph{Game of Life}~\cite{ar:Gardner1970}. This is a two-state game on the 2-dimensional integer lattice, but 
with a slightly richer set of rules. 
In a different context, these processes were considered by Granovetter~\cite{ar:Granovetter} as a model of 
the evolution of social influence.  There is certain resemblance with the class of processes that are known 
as~\emph{majority bootstrap processes}, but the crucial difference is that majority dynamics is \emph{non-monotone} 
in the sense that a vertex may change states multiple times. Thus, unlike the classical bootstrap processes, 
the process may never stabilise into a final configuration. 

However, as Goles and Olivos proved in~\cite{ar:GolOliv80}, if $G$ is finite, then eventually (that is, for $t$ 
sufficiently large) the process becomes periodic with period at most 2. More specifically, there is a $t_0$ 
depending on $G$ such that for any $t>t_0$ and for any $v\in V$ we have 
$S_{t}(v) = S_{t+2} (v)$. 

Majority dynamics is also a special case of voting with $q\geq 2$ alternative opinions, 
see~\cite{ar:MosselNeemanTamuz}. Each voter is assumed to be the vertex of a graph, and their initial 
opinions are selected from the set $\{1,\ldots, q\}$ independently of every other voter according to some 
distribution. At each round, a voter adopts the most popular opinion among its neighbours. 

In this paper we consider the evolution of majority dynamics on $\Gnp$, which is 
 the random graph on the set  $\VertexSet = [n]:=\{1,\ldots, n\}$, where 
 every pair of distinct vertices is present as an edge with probability $\EdgeProb$ independently of any other pair. 
We will consider this process on $\Gnp$ with initial configuration   
$\{S_0(v)\}_{v\in \VertexSet}$ being a family of independent random variables uniformly distributed in $\{ -1,+1\}$.
That is, each vertex in $\VertexSet$ initially is in state $+1$ with probability $1/2$, independently of the 
state of every other vertex.

Results regarding this setting were obtained recently by Benjamini et al.~\cite{ar:BenjChanmDonnelTamuz2016}. 
They showed that if $p\geq \lambda n^{-1/2}$ where $n > n_0$, for some sufficiently large constants $\lambda, n_0$, then 
$\Gnp$ is such that with probability at least 0.4 over the choice of the random graph and the choice of the 
initial state, the vertices in $\VertexSet$ unanimously hold the initially most popular state after four rounds. 
Benjamini et al. conjectured that in fact this holds with high probability.
The main result of this paper is the proof of their conjecture. 
\begin{theorem} \label{thm:main}
For all $0< \eps \le 1$ there exist $\lambda, n_0$ such that for all $n > n_0$, if $p\geq \lambda n^{-1/2}$, then 
$\Gnp$ is such that with probability at least 
$1-\eps$, over the choice of the random graph and the choice of the 
initial state, the vertices in $\VertexSet$ unanimously have state $\sgn(\sum_{v \in \VertexSet} S_0(v))$ 
after four rounds.
\end{theorem}

In our proof we exploit the fact that typically the initial number of vertices in the two states differ by at least $\Omega(\sqrt{n})$ vertices. Tran and Vu \cite{TranVu2019} showed that when $p$ is a constant, already a significantly smaller majority will lead to unanimity in 4 steps, with probability close to one. In particular, 
they showed that this happens already when one of the states exceeds the other by a large enough constant.

One might think that unanimity is reached for other classes of {\em sparser} random graphs or expanding graphs. 
However, Benjamini et al.~\cite{ar:BenjChanmDonnelTamuz2016} proved that for the class of 4-regular random 
graphs or 4-regular expander graphs, 
with high probability unanimity is not reached at any time, if the probability of state $+1$ in the beginning of 
the process is between $1/3$ and $2/3$. 
However, this is not the case for $d$-regular $\lambda$-expanders where $\lambda$ is the bound on the 
second-largest in absolute value eigenvalue, provided that $\lambda/d \leq 3/16$. Mossel et al.
(Theorem 2.3 in~\cite{ar:MosselNeemanTamuz}) showed that unanimity is reached eventually, under this assumption provided that the initial distribution 
of state $+1$ is sufficiently biased.
This bias is of order $1/\sqrt{d}$, that is, the assumption is that
 $\PPo{S_0(v)=+1}=\frac{1}{2} + \frac{c}{\sqrt{d}}$, for some constant $c>0$.

More recently, Zehmakan~\cite{ar:Zehm2018} proved a more general result on the evolution of 
majority dynamics on $d$-regular expander graphs. In particular, he proved that on a $d$-regular 
$\lambda$-expander graph $G$, when the initial configuration satisfies
$\sum_{v \in V(G)} S_0(v) \geq \frac{4\lambda}{d} n$, majority dynamics will reach 
the configuration where every vertex has state $+1$ within $O(\log_{d^2/\lambda^2} n)$ rounds.
Also, G\"artner and Zehmakan~\cite{ar:GartZeh2018} showed that if the initial density of the $-1$s is 
$1/2 - \eps$ for some $\eps>0$, then the majority dynamics will eventually reach the configuration where every vertex has state $+1$. 

Returning to the study of the process on $G(n,p)$ with $p = d/n$, Zehmakan~\cite{ar:Zehm2018} also 
showed that if $\PPo{S_0(v) = +1} 
=\frac12 + \omega \left(\frac{1}{\sqrt{d}}\right)$ and $d > (1+\eps )\log n$, then with high probability 
the process reaches unanimity in a constant number of rounds.

Another model of similar flavour is the model analysed by Abdullah and Draief~\cite{ar:AbdDraief} where 
instead of reading its entire neighbourhood, every vertex samples $k$ random vertices from its neighbourhood 
and adopts the state of the majority of the vertices in the random sample. Abdullah and Draief considered 
this model on  $G(n,p)$ with $p = d/n$ where $d \geq (2+\eps) \log n$.    
They showed that if the initial density of one of the two states is bounded away from $1/2$, then the above 
process will eventually arrive at unanimity with high probability. Moreover, the final is the one that has the initial majority.  

Besides the study of majority dynamics on random graphs, Benjamini et al.~\cite{ar:BenjChanmDonnelTamuz2016}
considered the question whether the Goles-Olivos theorem in~\cite{ar:GolOliv80}, which guarantees eventual periodicity for finite graphs, 
also holds for infinite graphs which satisfy certain assumptions. They showed that this is the case for the class of 
unimodular transitive graphs. These are vertex-transitive graphs (and therefore regular) in which flows that are 
invariant under the automorphism group are such that for every vertex the in-flow equals the out-flow. 
They also showed that stabilisation to periodicity occurs in at most $2d$ rounds, where $d$ is the degree of 
the graph. They conjectured that this is the case for every bounded degree infinite graph. 

Majority dynamics on other classes of graphs were recently considered by G\"artner and Zehmakan~\cite{ar:GartZeh2017}. They analysed majority dynamics on an $n\times n$ grid as well as 
on a torus. The initial state is determined by a random binomial subset of the vertex set, where every vertex 
is initially set to $-1$ with probability $p$ independently of every other vertex. 
 
The rest of the paper is devoted to the proof of Theorem~\ref{thm:main}. In Section~\ref{sec:outline} we present a  heuristic by Benjamini et al.~\cite{ar:BenjChanmDonnelTamuz2016} and  outline the proof of Theorem~\ref{thm:main}. 
We study the states after the first two rounds and the last two rounds in Sections~\ref{sec:begin} and~\ref{sec:end}, respectively. The proof of Theorem~\ref{thm:main} is presented in Section~\ref{sec:proof-mainthm}. 
We conclude the paper with discussions on a conjecture of Benjamini et al.~\cite{ar:BenjChanmDonnelTamuz2016} about smaller values of $d$.
 
\section{Heuristic and proof outline}\label{sec:outline}
It will be more convenient to use the average degree as the parameter instead of the edge probability. Set $d=np$ and the condition on $p$ in Theorem~\ref{thm:main} translates to $d\ge \lambda n^{1/2}$.

Let $\mu_t := n^{-1}\sum_{v \in \VertexSet} S_t (v)$ denote the average of the states of the vertices in $\VertexSet$ 
by step $t$. 
Benjamini et al.~\cite{ar:BenjChanmDonnelTamuz2016} conjectured that if $d \to \infty$, the 
quantities $(\mu_t^2)_{t\geq 0}$ increase with high probability  by a factor of $d$. More specifically, their heuristic is that 
$\mu_{t+1}^2 \gtrsim d \cdot \mu_t^2$, as long as $d \cdot \mu_t^2 \leq 1$. 

The heuristic is based on the assumption that redistributing the states of the vertices in every step does not alter the outcome significantly. More precisely, at the beginning of each step the state of every vertex follows an independent $\{-1,+1\}$-valued random variable with expectation $\mu_t$.
Then for a vertex $v$ which has $d$ neighbours the sum of the states of the neighbours behaves like $N(d\mu_t,d(1-\mu_t^2))$. Ignoring the possibility that there is no clear majority within the neighbourhood the probability that $v$ will have state $\sgn(\mu_t)$ in the following step is roughly $\PPo{N(d\mu_t,d)>0}\approx \frac{1}{2}+\sqrt{\frac{2}{\pi}}\sqrt{d}\mu_t$, when $\sqrt{d}|\mu_t| \ll 1$. 
As the $S_0(v)$s are independent and identically distributed on $\{-1,+1\}$, 
we have $\mathbb{E} [\mu_0^2] = \frac{1}{n}$.  
Hence, it is expected that the sequence $\mu_1^2, \mu_2^2,\ldots $ scales like $\frac{d}{n}, \frac{d^2}{n},\ldots$ 
until $d^t \approx n$. Thereafter, almost unanimity is reached in one more step, whereas one final step 
is required to arrive to complete unanimity.

The proof of Theorem~\ref{thm:main} is inspired by this heuristic. 
More precisely, the proof consists of two major parts, each consisting of two steps. In the first part (Lemma~\ref{lem:2-steps-lemma}) we show that with probability close to 1 almost every vertex adopts the state of the initial majority. Afterwards in the second part (Lemma~\ref{lem:lasttwo}) we prove that after two more steps, again with probability close to 1, every vertex will have the same state. 

For the first part (Section~\ref{sec:begin}), we will condition on the initial state satisfying 
$|\sum_{v \in V_n} S_0 (v)|\geq 2c \sqrt{n}$, for some $c >0$ such that 
the probability of this event is at least $1 - \eps/4$, for $n$ large enough. 
Then by using the second moment method on $X_2 (v)=\sum_{u \in N(v)} {\bf 1}(S_1 (u) =+1)$ we show that in two rounds an arbitrary 
vertex $v$ will have adopted the initial majority with probability $1-\eps/20$, when $n$ is large enough. 
For the second moment method we need to calculate the expectation (Lemma~\ref{lem:expectation}) and the variance (Lemma~\ref{lem:variance}) of the random variable  $X_2 (v)$. Finally, Markov's inequality implies that with probability at least 
$1-\eps/2$ most of the vertices have the same state as the initial majority. 

In the second part of the proof (Section~\ref{sec:end}) we will show that with probability $1-o(1)$ (as $n\to \infty$) 
(over the choices of the underlying graph) 
if we start with a configuration where 
all but at most $n/10$ of the vertices have state $+1$, then in two more steps all vertices will be of state $+1$ (Lemma~\ref{lem:lasttwo}). This will rely on an 
application of the union bound together with sharp concentration inequalities. 
Hence, the next 2 rounds will lead to unanimity.

\section{The first two rounds}\label{sec:begin}
In this section we will show that an arbitrary vertex will have state $+1$ after two rounds with probability close to 1,  
if we condition on an initial state with a sufficient majority of $+1$s. Due to symmetry this also holds for the state of a vertex to be $-1$ after two rounds, when the initial state has a sufficient majority of $-1$s.

In order to achieve this we will first expose the initial state of every vertex and only start exposing the edges afterwards. 
The following lemma ensures that after exposing the initial state of the vertices one of the two states will have a sufficient majority; the proof can be found in Section~\ref{sec:initial}.
\begin{lemma}[Initial state]\label{lem:initialstate}
Given $\eps>0$, set $c=c(\eps)=\sqrt{2\pi}\eps/20$. Then 
$$\PPo{ \big|\sum\nolimits_{u \in V_n} S_0 (u)\big |\geq 2c \sqrt{n} } \ge 1 - \eps/4,$$
when $n$ is large enough.
\end{lemma}
Throughout this section we will condition on the above event. Due to symmetry, we only need to consider the case when the initial state has a majority of $+1$s. Therefore we actually condition on the event $\{\sum_{u \in V_n} S_0 (u)\geq 2c \sqrt{n}\}$.
\begin{lemma}[First two rounds] \label{lem:2-steps-lemma_I} 
Given $\eps>0$, let $c=c(\eps)=\sqrt{2\pi}\eps/20$. For any $v\in V_n$, we have 
$$\PPo{S_2 (v) = +1 \mid \sum\nolimits_{u \in V_n} S_0(u) \geq 2 c\sqrt{n}} \ge 1- \eps/20,$$
when $n$ is large enough.
\end{lemma}

Roughly speaking, increasing the majority of $+1$s in the initial state should increase the probability that a vertex has state $+1$ in any step of the process. This in turn should imply that if we can establish a lower bound on the probability of $S_2(v)=+1$ conditional on the smallest possible value of $\sum\nolimits_{u \in V_n} S_0(u)$ satisfying $\sum\nolimits_{u \in V_n} S_0(u)\geq 2 c\sqrt{n}$,  then this will also hold when conditioning on the whole range. We establish the lower bound in Lemma~\ref{lem:2-steps-lemma} and apply it to prove Lemma~\ref{lem:2-steps-lemma_I}. 
Since $\sum\nolimits_{u \in V_n} S_0(u)$ takes only integer values the smallest possible value of $\sum\nolimits_{u \in V_n} S_0(u)$  satisfying $\sum\nolimits_{u \in V_n} S_0(u)\geq 2 c\sqrt{n}$ would potentially be  $\Bigl\lceil 2 c\sqrt{n}\Bigr\rceil$, but we also need to take into account that $\sum_{u\in \VertexSet}S_0(u)$  only takes integer values between $-n$ and $n$ which have the same parity as $n$. 
To reflect this, we let $\varphi(x)$ denote the equivalent of the ceiling function, taking the mod 2 equivalence into account; formally we define $\varphi(x):=\min_{k\in \mathbb{Z}}\{k\ge x \ \text{and} \ k\equiv n \mod 2\}$.  Thus, in Lemma~\ref{lem:2-steps-lemma} we should condition on the event that $\sum\nolimits_{u \in V_n} S_0(u) =  \varphi(2 c\sqrt{n})$. But to ease notation,  we will omit the function $\varphi$ for the remainder of the paper.

Let $\Ev{c}$ denote the event that 
$$\sum\nolimits_{u \in V_n} S_0 (u)= 2c \sqrt{n}.$$
\begin{lemma} \label{lem:2-steps-lemma}
Given $\eps>0$, set $c=c(\eps)=\sqrt{2\pi}\eps/20$. 
For any $v\in V_n$, we have 
$$\PPo{S_2 (v) = +1 \mid \Ev{c}} \ge 1- \eps/20,$$
when $n$ is large enough.
\end{lemma}
We defer the proof to Section~\ref{sec:2-steps}. Now we can prove Lemma~\ref{lem:2-steps-lemma_I}.
\begin{proof}[Proof of Lemma~\ref{lem:2-steps-lemma_I}]
The result follows from Lemma~\ref{lem:2-steps-lemma} once we show 
\begin{equation}\label{eq:probincrease}
\PPo{S_2 (v) = +1 \mid \sum\nolimits_{u \in V_n} S_0(u)=k}\le \PPo{S_2 (v) = +1 \mid \sum\nolimits_{u \in V_n} S_0(u)=k+2}
\end{equation}
for every integer $-n\le k \le n-2$ with $k\equiv n \mod 2$. For some fixed $-n\le \ell \le n$ with $\ell\equiv n \mod 2$, let $\rt{\ell}$ be an arbitrary initial configuration compatible with $\sum\nolimits_{u \in V_n} S_0(u)=\ell$. As we have not exposed any edges until this point, due to symmetry, we have 
$$\PPo{S_2 (v) = +1 \mid \sum\nolimits_{u \in V_n} S_0(u)=\ell}=\PPo{S_2 (v) = +1 \mid S_0(u)=\rt{\ell}}.$$
Now select $\rt{k}$ and $\rt{k+2}$ in such a way that $\rt{k+2}$ can be created from $\rt{k}$ by changing the initial state of one vertex from $-1$ to $+1$. Note that in any fixed graph on $\VertexSet$, changing the state of vertices from $-1$ to $+1$ in a given step, by the monotonicity of the $\sgn$ function, can only result in changing the state of vertices from $-1$ to $+1$ in the following step. Using this argument repeatedly for the first two steps implies \eqref{eq:probincrease} and the result follows.
\end{proof}

The remainder of this section is structured as follows. We prove Lemma~\ref{lem:initialstate} in Section~\ref{sec:initial}. Then we prove Lemma~\ref{lem:2-steps-lemma}, subject to two technical lemmas (Lemmas~\ref{lem:expectation} and \ref{lem:variance}) in Section~\ref{sec:2-steps}. Finally we prove Lemmas~\ref{lem:expectation} and \ref{lem:variance} in Section~\ref{sec:rounds2} and \ref{sec:var2}, respectively.

Before proceeding with the proof of Lemma~\ref{lem:initialstate}, we introduce auxiliary notations, which will be used in Sections~\ref{sec:2-steps}--\ref{sec:var2}. 
Starting with the proof of Lemma~\ref{lem:2-steps-lemma} in Section~\ref{sec:2-steps}, we consider an initial configuration $\st{0}$ 
compatible with $\Ev{c}$, and condition on the event $\mathcal S_0:=\{ S_0 = \st{0}\}$, where $S_0$ consists of $S_0(u)$ for every $u\in \VertexSet$. 
Now explore the neighbourhood of $v$, which we denote by $N(v)$. 
We also condition on the event $\mathcal{N}_\Gamma(v):=\{N(v)=\Gamma\}$ 
for some fixed set $\Gamma\subseteq V\setminus \{v\}$.
With abuse of notation we write $\mathcal S_0 \cap \mathcal N_\Gamma(v)$ for the intersection of these events. 
Throughout Sections~\ref{sec:2-steps}--\ref{sec:var2}, we will work on the conditional space $\mathcal S_0 \cap \mathcal N_\Gamma(v)$.

\subsection{The initial state: proof of Lemma~\ref{lem:initialstate}}\label{sec:initial}

Lemma~\ref{lem:initialstate} is a consequence of the following local limit theorem for a binomial random variable of the form  $\mathrm{Bin} (k,q(k))$  as we shall see below.%
\begin{theorem}[Local Limit Theorem]\label{cor:lll}
	There exists an absolute constant $\gamma$ such that for every positive integer $k$ and every function $0<q(k)<1$ the random variable $X \sim \mathrm{Bin} (k,q(k))$ satisfies 
	$$\sup_{i \in \mathbb{N}\cup \{0\}}\left|\sqrt{\Var{X}} \cdot \PPo{X=i}-\frac{1}{\sqrt{2\pi}}\exp\left(-\frac{(i-\EE[X])^2}{2\Var{X}}\right)\right| <\frac{\gamma}{\sqrt{\Var{X}}}.$$
\end{theorem}

Note that Theorem~\ref{cor:lll} is about the distribution of the sum of $k$ independent Bernoulli-distributed 
random variables whose parameters may depend on $k$. 
It is a generalisation of a classical result on a local limit theorem 
for partial sums of infinite sequences of independent random variables variables (e.g. Theorem 4 in Chapter VII from~\cite{MR0388499}). 
Its proof can be found in Section~\ref{sec:proof_of_LLT}.

\begin{proof}[Proof of Lemma~\ref{lem:initialstate}]
For $\eps>0$ recall that we set $c=c(\eps)=\sqrt{2\pi}\eps/20$. By Theorem~\ref{cor:lll} with $X=\Bin{n,1/2}$, and as $\Var{\Bin{n,1/2}}=n/4$, we have
$$\PPo{\Big|\sum\nolimits_{v \in V_n} S_0 (v)\Big|< 2 c\sqrt{n}}\le \left(2c\sqrt{n}+1\right)\left(\frac{2}{\sqrt{2\pi n}}+\frac{4 \gamma}{n}\right)\le \frac{\eps}{5} + o(1) \le \frac{\eps}{4},$$
where the last inequality holds for sufficiently large $n$. 
\end{proof}

\subsection{The first two steps: proof of Lemma~\ref{lem:2-steps-lemma}}\label{sec:2-steps}

Recall that for the remainder of the section we work on the conditional space $\mathcal S_0 \cap \mathcal N_{\Gamma}(v)$. 
In particular, we consider the family $\{ S_1 (u) \}_{u \in N(v)}$, conditional on $\{ S_0 = \st{0}\}$, where $\st{0}$ is compatible with $\Ev{c}$, and 
a certain realisation of $N(v)$. 
To derive Lemma~\ref{lem:2-steps-lemma}, we will show that, uniformly over the choice of $\st{0}$ and $\Gamma\subseteq V\setminus \{v\}$ with $\big||\Gamma|-d\big|\le d^{2/3}$, we have 
$$\PPo{S_2 (v) = +1 \mid  \mathcal S_{0}\cap \mathcal N_{\Gamma}(v)} \ge 1- \eps/20.$$
In particular,  we will apply a second moment argument on the random variable 
$$X_2 (v) = \sum\nolimits_{u \in N(v)} {\bf 1}(S_1 (u) =+1),$$ 
conditioned on $\mathcal S_{0}\cap  \mathcal N_{\Gamma}(v)$.
(Note that in this conditional space if $X_2 (v ) > |\Gamma|/2$, then $S_2(v) = +1$.) 
To this end, we obtain bounds on the expectation and the variance of $X_2 (v)$. 

\begin{lemma}\label{lem:expectation}
	There exists a constant $\xi$ (independent of $\eps$) such that for large enough $n$ and any $\Gamma\subseteq V\setminus \{v\}$ satisfying $\big||\Gamma|-d\big|\le d^{2/3}$ we have
	$$\Ex{X_2 (v) \mid \ \mathcal S_{0}\cap \mathcal N_{\Gamma}(v) }\ge \frac{|\Gamma|}{2}+\frac{\xi c}{7} \cdot \left( \frac{d^3}{n} \right)^{1/2}.$$
\end{lemma} 
\begin{lemma}\label{lem:variance}
	Let $\gamma$ be as in Theorem~\ref{cor:lll}. Then for any $\Gamma\subseteq V\setminus \{v\}$ satisfying $\big||\Gamma|-d\big|\le d^{2/3}$ we have
	$$\Var{X_2 (v) \ \mid \mathcal S_{0}\cap \mathcal N_{\Gamma}(v)} \le (\max\{96\gamma^2,8\}+1)d.$$ 	
\end{lemma}
We defer the proof of these two lemmas to Sections~\ref{sec:rounds2} and~\ref{sec:var2}, respectively. We now prove Lemma~\ref{lem:2-steps-lemma} using them. In the following proof as well as later, we will use 
\begin{equation}\label{eq:lambdadef}
\lambda=\lambda(\eps)=\max\{c^{-2},\beta^2\},
\end{equation}
where $\beta$ is a large constant independent of $\eps$.  

\begin{proof}[Proof of Lemma~\ref{lem:2-steps-lemma}]
	Let $\gamma'=\max\{96\gamma^2,8\}+1$. 
	By Lemma~\ref{lem:expectation} we have for any $\Gamma\subseteq V\setminus \{v\}$ satisfying $\big||\Gamma|-d\big|\le d^{2/3}$
	$$ \Ex {X_2 (v) \mid \mathcal S_0 \cap \mathcal{N}_{\Gamma} (v)  } - \frac{|\Gamma|}{2}> 
	\frac{\xi c}{7} \cdot \left( \frac{d^3}{n} \right)^{1/2},$$
	Chebyshev's inequality implies that 
	\begin{equation*}
	\begin{split}
	\PPo{ X_2 (v) \le  \frac{|\Gamma|}{2} \mid \mathcal S_0 \cap \mathcal{N}_{\Gamma} (v)} &\le \frac{49\gamma'}{\xi^2}\cdot \frac{d}{c^2d^3/n} = \frac{49\gamma'}{\xi^2}\cdot \frac{n}{c^2d^2}\\
	&\le \frac{49\gamma'}{\xi^2}\cdot \frac{1}{c^2 \lambda^2} \stackrel{\eqref{eq:lambdadef}}{\le} \frac{49\gamma'}{\xi^2}\cdot \frac{1}{c^2 (c^{-3}\beta)}\stackrel{c=\sqrt{2\pi}\eps/20}{\le} \frac{\eps}{40},
	\end{split}
	\end{equation*}
	when $\beta$ is large enough. In particular this implies that
	$$\PPo{ S_2(v)=+1 \mid \mathcal S_0 \cap \mathcal{N}_{\Gamma} (v)}\ge 1-\frac{\eps}{40},$$
	and as this holds for any $\Gamma\subseteq V\setminus \{v\}$ satisfying $\big||\Gamma|-d\big|\le d^{2/3}$ we also have
	$$\PPo{ S_2(v)=+1 \mid \mathcal S_0 \cap \mathcal{N} (v)}\ge 1-\frac{\eps}{40},$$
	where 	
	$$\mathcal{N} (v):=\bigcup_{\substack{\Gamma \subseteq V\setminus \{v\}\\ \big||\Gamma|-d\big|\le d^{2/3}}}\mathcal{N}_{\Gamma}(v)=\{\big|\ |N(v)| - d\ \big| \le d^{2/3}\}.$$ Note that (by a standard Chernoff bound) the event $\mathcal{N} (v)$ holds with probability $1- \exp (-\Theta (d^{1/3}))\ge 1-\eps/40$, for large enough $n$, and that it is independent of $\mathcal S_0$. Therefore, we obtain
	$$\PPo{ S_2(v)=-1 \mid \mathcal S_0}\le \PPo{ S_2(v)=-1 \mid \mathcal S_0 \cap \mathcal{N} (v)}+1-\PPo{\mathcal{N} (v)}\le \frac{\eps}{20}.$$
	This yields $\PPo{S_2 (v) = +1 \mid \Ev{c}} \ge 1- \eps/20$.
\end{proof}

\subsection{The expectation  of $X_2(v)$: proof of Lemma~\ref{lem:expectation}} \label{sec:rounds2}
Fix $\Gamma\subseteq V\setminus \{v\}$ satisfying $\big||\Gamma|-d\big|\le d^{2/3}$.
Recall that throughout this section we are working on the conditional space $\mathcal S_{0}\cap \mathcal N_{\Gamma}(v)$.
So to ease notation we will drop this conditioning from the probabilities in this as well as in the next section.
Consider the set $V_n \setminus \{v,u\}$ and split it into three parts $V_+,V_-$ and $V_{++}$ such that $V_+ \cup V_{++}$ is the set of vertices with initial state $+1$, while $V_-$ is the set of vertices with initial state $-1$ and in addition $|V_+|=|V_-|$. Note that $|V_+|=|V_-|=\frac{1}{2} (n - 2c\sqrt{n}) - {\bf 1}(S_0 (u) =-1) -{\bf 1}(S_0 (v) =-1)$ and $|V_{++}|=(n-2) - 2|V_+|=2c\sqrt{n} +2({\bf 1}(S_0 (u) =-1) +{\bf 1}(S_0 (v) =-1))-2$.

Clearly, we have 
\begin{equation} \label{eq:0th_lowerbound}
\PPo{S_1 (u) =+1} \geq 
\PPo { |N(u) \cap V_+| + |N(u) \cap V_{++}| \ge |N(u) \cap V_-| + 2},
\end{equation}
as the latter is the probability that $S_1 (u)=+1$ under the assumption that  
$\st{0} (u) = \st{0} (v) = -1$. 

For brevity, we set 
$$n_+(u)=|N(u) \cap V_+|,\ n_{++}(u)=|N(u) \cap V_{++}|, \text{ and } n_-(u)= |N(u) \cap V_-|.$$

We will bound \eqref{eq:0th_lowerbound} from below, conditioning on the value of $n_{++} (u)$, and 
we are going to consider several cases depending on the range of this value. 

Note that $n_+(u), n_-(u) \sim \Bin{|V_+|,\frac{d}{n}}$ and $\ n_{++}(u)\sim \Bin{|V_{++}|,\frac{d}{n}}$. Thus
$$ \Ex{n_+(u)}, \Ex{n_-(u)} =\Theta (d), \ \mbox{whereas} \ \Ex{n_{++} (u)}= \frac{2cd(1+o(1))}{\sqrt{n}}.$$

Set $\mm{u}:= \Ex{n_{++} (u)}$. The proof hinges on $\mm{u}$ being large enough. This is achieved by requiring the average degree to be large enough, in fact this is the only point in the proof of Lemma~\ref{lem:2-steps-lemma_I} where this condition is required. By the choice of $\lambda=\max\{c^{-2},\beta^2\}$\, for $n$ sufficiently large we have 
\begin{equation}\label{eq:expgain}
\mm{u}=\frac{2cd(1+o(1))}{\sqrt{n}} \ge c\lambda \geq \lambda^{1/2}\ge \beta.
\end{equation}

We write 
\begin{equation}\label{eq:Sigmas}
\PPo {n_+(u) + n_{++}(u) \ge n_- (u) + 2}  
\ =\ \Sigma_0 +\Sigma_1 + \Sigma_2+\Sigma_3,
\end{equation}
where 
\begin{equation*}
\begin{split}
&\Sigma_0:=\sum_{0\leq k \leq3} \PPo{n_{++}(u) =k} \cdot 
\PPo {n_+(u) + k \geq n_- (u) + 2}, \\
&\Sigma_1:=\sum_{4\leq k \leq \mm{u}/2} \PPo{n_{++}(u) =k} \cdot 
\PPo {n_+(u) + k \geq n_- (u) + 2}, \\
&\Sigma_2:=\sum_{\mm{u}/2 < k \leq 2 \mm{u}} \PPo{n_{++}(u) =k} \cdot 
\PPo {n_+(u) + k \geq n_- (u) + 2},\\
&\Sigma_3:=\sum_{2\mm{u} < k } \PPo{n_{++}(u) =k} \cdot 
\PPo {n_+(u) + k \geq n_- (u) + 2}.
\end{split}
\end{equation*}

We first derive a lower bound on $\Sigma_0$.
\begin{claim} \label{clm:first_6_terms}
	We have 
	$$\Sigma_0
	\ \geq \ \frac12 \cdot \sum_{k=0}^{3} \PPo{n_{++} (u) = k}. $$
\end{claim}

\begin{proof}[Proof of Claim~\ref{clm:first_6_terms}]
	Recall that $n_+(u), n_- (u)$ are identically distributed. Therefore, for any integer $\alpha$ 
	we can write 
	\begin{equation*}
	\begin{split} 
	\PPo{n_+(u) \geq n_- (u) + 1 + \alpha} &= 1- \PPo{n_+(u) < n_- (u) + 1 + \alpha}\\
	&= 1- \PPo{n_+(u) \leq n_- (u) + \alpha} \\
	&=1- \PPo{n_-(u) \leq n_+ (u) + \alpha}.
	\end{split}
	\end{equation*} 
	This can be re-written as
	$$ \PPo{n_+(u)+ 1 - \alpha \geq n_- (u) + 2} + \PPo{n_+(u) + 2+\alpha \geq n_- (u) +2} =1.
	$$
	Also, note that for $\alpha \geq 0$ we have 
	$$  \PPo{n_+(u)+ 1 - \alpha \geq n_- (u) + 2} <  
	\PPo{n_+(u) + 2+\alpha \geq n_- (u) +2}. 
	$$
	Thereby, we can write 
	$$ \PPo{n_+(u) + 2+\alpha \geq n_- (u) +2} = \frac12 + s_\alpha, $$
	for some $s_\alpha > 0$ when $\alpha\ge 0$, whereby 
	$$  \PPo{n_+(u)+ 1 - \alpha \geq n_- (u) + 2}  = \frac12 - s_\alpha.$$
		Using these, we can write 
	\begin{equation*}
	\begin{split} 
	&\PPo{n_{++}(u) = 1- \alpha} \cdot 
	\PPo{n_+(u)+ 1 - \alpha \geq n_- (u) + 2} +  \\
	& \hspace{3cm}
	\PPo{n_{++}(u) = 2 +\alpha} \cdot 
	\PPo{n_+(u) + 2+\alpha \geq n_- (u) +2} =  \\
	& \PPo{n_{++}(u) = 1- \alpha} \cdot 
	\left( \frac12 - s_\alpha \right) +
	\PPo{n_{++}(u) = 2 +\alpha} \cdot 
	\left(\frac12 + s_\alpha \right) \\ 
	&=\frac12\cdot  \left (\PPo{n_{++}(u) = 1- \alpha} + \PPo{n_{++}(u) = 2 +\alpha}  
	\right) \\
	&\hspace{3cm}+ s_\alpha \left(\PPo{n_{++}(u) = 2 +\alpha}  -  \PPo{n_{++}(u) = 1- \alpha} \right).
	\end{split}
	\end{equation*}
	But $s_\alpha > 0$ and when $0\le \alpha<2$ (or equivalently $\alpha\in \{0,1\}$), for $\beta$ large enough by \eqref{eq:expgain} we have $2+\alpha<\mm{u}$. Thus 
	$\PPo{n_{++}(u) = 2 +\alpha}  > \PPo{n_{++}(u) = 1- \alpha}$, whereby we conclude that the second summand is positive. 
	Hence for $0\le \alpha<2$ we have,
	\begin{equation} \label{eq:lower_bound_term}
	\begin{split} 
	&\PPo{n_{++}(u) = 1- \alpha} \cdot 
	\PPo{n_+(u)+ 1 - \alpha \geq n_- (u) + 2} +  \\
	& \hspace{3cm}
	\PPo{n_{++}(u) = 2 +\alpha} \cdot 
	\PPo{n_+(u) + 2+\alpha \geq n_- (u) +2} \\
	&> \frac12\cdot  \left (\PPo{n_{++}(u) = 1- \alpha} + \PPo{n_{++}(u) = 2 +\alpha}  
	\right).
	\end{split}
	\end{equation}
	Now, we pair up the four terms of $\Sigma_0$ (for $k\in\{0,1,2,3\}$) using the value of $\alpha$. 
	In particular, $\alpha = 0$ corresponds to $k\in\{1, 2\}$, and $\alpha =1$ corresponds to $k\in \{0, 3\}$. 
	In other words, we write 
	\begin{equation*}
	\begin{split}
	&\Sigma_0=\sum_{k=0}^{3} \PPo{n_{++}(u) =k} \cdot 
	\PPo {n_+(u) + k \geq n_- (u) + 2}\\ 
	&= \sum_{\alpha =0}^1 \left( \PPo{n_{++}(u) = 1- \alpha} \cdot 
	\PPo{n_+(u)+ 1 - \alpha \geq n_- (u) + 2} +  \right. \\
	& \hspace{3cm}
	\left. \PPo{n_{++}(u) = 2 +\alpha} \cdot 
	\PPo{n_+(u) + 2+\alpha \geq n_- (u) +2} \right) \\ 
	&\stackrel{\eqref{eq:lower_bound_term}}{>} 
	\frac12 \sum_{\alpha =0}^1 \left (\PPo{n_{++}(u) = 1- \alpha} + \PPo{n_{++}(u) = 2 +\alpha}  \right) \\
	& =\frac12 \sum_{k=0}^3\PPo{n_{++}(u) =k},
	\end{split}
	\end{equation*}
	which concludes the proof of the claim. 
\end{proof}
To obtain a lower bound on $\Sigma_1$, we use the following simple fact that for any integer $k \geq 0$
\begin{equation} \label{eq:identical_binomials}
\PPo{n_+(u) + k \geq n_-(u)} > 1/2.
\end{equation}
To see this, note that since
\begin{equation*}
\begin{split}
\PPo{n_+ (u) +k \geq n_- (u)} + \PPo{n_+ (u) + k < n_- (u)}= 1,
\end{split}
\end{equation*}
the result follows if $\PPo{n_+ (u) +k \geq n_- (u)} > \PPo{n_+ (u) + k < n_- (u)}$.
Since $n_+ (u)$ and $n_- (u)$ are identically distributed, 
$$ \PPo{n_+ (u) +k < n_- (u)} = \PPo{n_- (u) +k < n_+ (u)}.$$ 
But also since $k\ge 0$, we have
$$  \PPo{n_- (u) +k < n_+ (u)} < \PPo{n_- (u) \leq n_+ (u)+k},$$
and~\eqref{eq:identical_binomials} follows. 

Therefore, we have
\begin{equation}\label{eq:Sigma_1_lower_bound}
\begin{split}
\Sigma_1 &=\sum_{4\leq k \leq \mm{u}/2} \PPo{n_{++}(u) = k} \cdot \PPo{n_+ (u) + k \geq n_- (u) +2} \\
&> 
\frac12 \cdot \sum_{4\leq k \leq \mm{u}/2} \PPo{n_{++} (u) = k}. 
\end{split}
\end{equation}
An analogous argument implies
\begin{equation}\label{eq:Sigma_3_lower_bound}
\Sigma_3 > \frac12 \cdot \sum_{2\mm{u}< k} \PPo{n_{++} (u) = k}.
\end{equation}

We now turn to $\Sigma_2$, and start by providing a bound on $\PPo{n_+(u) + \ell \geq n_- (u)}$.

\begin{claim}\label{clm:boost} When $n$ is large enough, for every $\ell$ with $\mm{u}/2-2< \ell \le 2\mm{u}-2$ there exists a constant $\xi$ independent of $\eps$ such that
	$$ \PPo{n_+(u) + \ell \geq n_- (u)} > \frac12 + \xi \cdot \frac{\ell}{\sqrt{d}}.
	$$
\end{claim}
\begin{proof}
We start by considering the case when $d>n/(1+c^2/162)$.
Note that under this assumption and by \eqref{eq:expgain} we have, for large enough $n$,
\begin{equation}\label{eq:smallvar}
\Var{n_+(u)}^{1/2}\leq \sqrt{\frac{n}{2}\frac{d}{n}\left(1-\frac{d}{n}\right)}\le \frac{cd}{18n^{1/2}}\le\frac{1}{32}\mm{u}.
\end{equation}
When $n$ is large enough, we have $ \mm{u}/2-2\ge \mm{u}/4$.
Since $n_+(u)$ and $n_-(u)$ are identically distributed and independent, for any $\ell\ge \mm{u}/2-2\ge \mm{u}/4$  we have
\begin{equation*}
\begin{split}
&\PPo{n_+(u) + \ell \geq n_- (u)} \\
&\hspace{10ex}\ge \PPo{n_+(u) \ge \Ex{n_+(u)}-\mm{u}/8 }\PPo{n_-(u) \le \Ex{n_-(u)}+\mm{u}/8}\\
&\hspace{10ex}\ge \left(1-\exp\left(-\frac{\mm{u}^2}{128(\Var{n_+(u)}+\mm{u}/24)}\right)\right)^2,
\end{split}
\end{equation*}
where the last step follows from the Chernoff bound. Together with \eqref{eq:smallvar} this implies
\begin{equation*}
\begin{split}
\PPo{n_+(u) + \ell \geq n_- (u)}&\ge \left(1-\exp\left(-\frac{\mm{u}^2}{128((\mm{u}/32)^2+\mm{u}/24)}\right)\right)^2\\
&\ge \left(1-e^{-2}\right)^2,
\end{split}
\end{equation*}
for large enough $\beta$. The claim follows as $\left(1-e^{-2}\right)^2>1/2$ and $\ell\le 2\mm{u}\le \sqrt{d}$.

Now assume $d\le n/(1+c^2/162)$. Note that in this case, when $n$ is large enough,
\begin{equation}\label{eq:largevar}
\Var{n_+(u)}^{1/2}=\sqrt{\frac{n(1-o(1))}{2}\frac{d}{n}\left(1-\frac{d}{n}\right)}\ge \frac{cd}{20n^{1/2}}\ge \frac{1}{42}\mm{u}.
\end{equation}

For any positive integer $\ell$ we  write 
\begin{equation}\label{eq:binincrease}
\begin{split}
&\PPo{n_+(u) + \ell \geq n_- (u)} =
\PPo{n_+ (u) \geq n_-(u)} + \sum_{i=1}^{\ell} \PPo {n_+ (u) + i = n_-(u)}. 
\end{split}
\end{equation}
By~\eqref{eq:identical_binomials}, we obtain
$$ \PPo{n_+ (u) \geq n_- (u)} > \frac12. $$
This together with \eqref{eq:binincrease} gives
\begin{equation}\label{eq:binshiftlower}
\PPo{n_+(u) + \ell \geq n_- (u)} >
\frac12 + \sum_{i=1}^{\ell} \PPo {n_+ (u) + i = n_-(u)}. 
\end{equation}
To bound the terms of the sum from below, we condition on the value of $n_- (u)$ to obtain
\begin{equation*}
\begin{split}
\PPo {n_+ (u) + i = n_-(u)} \geq 
\sum_{s\in \big[ \Ex{n_-(u)} \pm 2\Var{n_- (u)}^{1/2}\big]}
\PPo{n_- (u) =s} \cdot \PPo{n_+ (u) = s -i}. 
\end{split}
\end{equation*}
Note that both $n_+(u)$ and $n_-(u)$ are binomially distributed with the same parameters.  
By Theorem~\ref{cor:lll} there exists $\xi'>0$ such that for any $s \in [\Ex{n_+(u)} \pm 2\Var{n_+ (u)}^{1/2}]$ and for any $i=1,\ldots, \ell$, where $\ell \leq 2\mm{u}\stackrel{\eqref{eq:largevar}}{\le} 84\Var{n_+ (u)}^{1/2}$, we have 
\begin{equation*}
\PPo{n_+ (u) = s-i} \geq \frac{\xi'}{\sqrt{d}}.
\end{equation*}
Therefore, since $n_+(u), n_-(u) \sim \Bin{|V_+|,\frac{d}{n}}$, there exists $\xi>0$ such that
\begin{equation*}
\begin{split}
\PPo {n_+ (u) + i = n_-(u)}
&\ge \sum_{s\in  \big[\Ex{n_-(u)} \pm 2\Var{n_- (u)}^{1/2} \big]} 
\PPo{n_- (u) =s} \cdot \PPo{n_+ (u) = s -i}\\
&\ge \frac{\xi'}{\sqrt{d}} \sum_{s\in  \big[\Ex{n_-(u)} \pm 2\Var{n_- (u)}^{1/2}\big]} \PPo{n_- (u) =s} 
\ge \frac{\xi}{\sqrt{d}}, 
\end{split}
\end{equation*}
 where the last inequality follows from Chebyshev's inequality. 
 
Together with \eqref{eq:binshiftlower}, for any such $\ell$ we have 
$$ \PPo{n_+(u) + \ell \geq n_- (u)} > \frac12 + \xi \cdot \frac{\ell}{\sqrt{d}}.
$$
\end{proof}

Now we will use Claim~\ref{clm:boost} to derive a lower bound on $\Sigma_2$:
\begin{equation} \label{eq:Sigma2_LB}
\begin{split}
\Sigma_2 &= \sum_{\mm{u}/2 < k \leq 2\mm{u}} \PPo{n_{++}(u) =k} \cdot 
\PPo {n_+(u) + k \geq n_- (u) + 2} \\
&> \frac12 \cdot \sum_{\mm{u}/2 < k \leq 2\mm{u}} \PPo{n_{++}(u) =k} + 
\frac{\xi}{\sqrt{d}} \cdot \sum_{k=\mm{u}/2+1}^{2\mm{u}} \PPo{n_{++}(u) =k} \cdot (k-2).
\end{split}
\end{equation}
We will show that the second sum is close to $\mm{u}$, which is $(1+o(1))2cd/\sqrt{n}$ by \eqref{eq:expgain}.
This will imply that the second summand is of order $c\sqrt{d/n}$.

Clearly, we have
\begin{equation*}
\begin{split}
 &\sum_{\mm{u}/2 < k \leq 2\mm{u}} \PPo{n_{++}(u) =k} \cdot (k-2) \\
 &\hspace{10ex}\ge (\mm{u}/2 - 2) \cdot \PPo{\mm{u}/2< n_{++} (u) \le 2\mm{u}}.
 \end{split}
 \end{equation*}
By the Chernoff bound we have for large enough $\beta$ 
$$\PPo{\mm{u}/2< n_{++} (u) \le 2\mm{u}}\ge 1-2\exp\left(- \frac{cd}{8\sqrt{n}}\right)\stackrel{\eqref{eq:expgain}}{\ge}1-2\exp\left(- \frac{\beta}{18}\right)>\frac{1}{2}.$$ 
So for large enough $\beta$ 
$$ \sum_{\mm{u}/2 < k \leq 2\mm{u}} \PPo{n_{++}(u) =k} \cdot (k-2)  > \frac{1}{6} \mm{u}.$$
Substituting this into~\eqref{eq:Sigma2_LB} we deduce that 
\begin{equation} \label{eq:Sigma2_LBFinal}
\Sigma_2 > \frac12 \cdot \sum_{k=\mm{u} \leq k\leq 2\mm{u}} \PPo{n_{++}(u) =k} + 
\frac{\xi c}{6} \cdot \left( \frac{d}{n} \right)^{1/2}.
\end{equation}
Therefore, Claim~\ref{clm:first_6_terms}, \eqref{eq:Sigma_1_lower_bound}, \eqref{eq:Sigma_3_lower_bound}, and \eqref{eq:Sigma2_LBFinal} in \eqref{eq:Sigmas} give 
\begin{equation*}
\PPo {n_+(u) + n_{++}(u) \ge n_- (u) + 2} >
\frac12 \cdot \sum_{0\leq k} \PPo{n_{++}(u) =k} + 
\frac{\xi c}{6} \cdot \left( \frac{d}{n} \right)^{1/2}.
\end{equation*}
By \eqref{eq:0th_lowerbound} and because $n_{++}$ is a non-negative integer we have
\begin{equation} \label{eq:prob_step1}
\PPo{S_1 (u) =+1 \mid \ \mathcal S_{0}\cap \mathcal N_{\Gamma}(v)}\geq 
\PPo {n_+(u) + n_{++}(u) \ge n_- (u) + 2} >
\frac12 + \frac{\xi c}{6} \cdot \left( \frac{d}{n} \right)^{1/2}.
\end{equation}
Now, $d - d^{2/3}  > 6d/7$ for $n$ sufficiently large. 
So~\eqref{eq:prob_step1} yields  
\begin{equation*} \Ex {X_2 (v)} > \frac{|\Gamma|}{2} 
+ \frac{\xi c}{7} \cdot \left( \frac{d^3}{n} \right)^{1/2},
\end{equation*}
completing the proof of Lemma~\ref{lem:expectation}.

\subsection{The variance of $X_2(v)$: proof of Lemma~\ref{lem:variance}}\label{sec:var2}

We will now bound the variance of  $X_2 (v) = \sum_{u \in N(v)} {\bf 1}(S_1 (u) =+1)$ conditional on $\mathcal S_0\cap \mathcal N_{\Gamma}(v)$, for any $\Gamma\subseteq V\setminus \{v\}$ satisfying $\big||\Gamma|-d\big|\le d^{2/3}$. Recall that due to the conditioning we have revealed the initial state of every vertex and the edges adjacent to $v$, however we have not examined any further edges so far.

We let $I_u ={\bf 1}(S_1 (u) =+1)$, for all $u \in N(v)$ and write 
$$\Var{X_2 (v)} = \sum\nolimits_{u,u' \in N(v)} \Cov{I_u}{I_{u'}}.$$ 
Let $E$ denote the edge set of $G(n,p)$.

\begin{claim} \label{clm:cov} 
For $u,u' \in N(v)$, such that $u\neq u'$, we have 
 \begin{equation}\label{eq:covfin}
 \begin{split}
\Cov{I_u}{I_{u'}} 
&=p(1-p)\left(\PPo{I_u=1\mid uu' \in E}-\PPo{I_u=1\mid uu' \not \in E}\right) \times \\
&\hspace{0.5cm} \left(\PPo{I_{u'}=1 \mid uu'\in E}  -  \PPo{I_{u'} =1 \mid uu'\not \in E } \right).
\end{split}
\end{equation}
\end{claim}
\begin{proof}[Proof of Claim~\ref{clm:cov}]
Consider first two distinct vertices $u,u'\in N(v)$. We have
\begin{align}
\Cov{I_u}{I_{u'}}
& = \Ex{I_u I_{u'}} - \Ex{I_u}\cdot \Ex{I_{u'}} \nonumber\\ 
&=  \PPo{I_{u}=1, I_{u'}=1}-  \PPo{I_{u}=1} \cdot \PPo{I_{u'}=1}. \label{eq:cov0}
\end{align}

The first term of \eqref{eq:cov0} can be rewritten as
\begin{align*}
\PPo{I_u=1,I_{u'}=1} 
&=  \PPo{I_u=1, I_{u'}=1 \mid uu' \in E} \cdot \PPo{uu'\in E} \nonumber\\ 
&\hspace{2.5cm}+\PPo{I_u=1,I_{u'}=1 \mid uu' \not \in E} \cdot \PPo{uu' \not \in E} \nonumber\\
&=  \PPo{I_u=1, I_{u'}=1 \mid uu' \in E} \cdot p +\PPo{I_u=1,I_{u'}=1 \mid uu' \not \in E} \cdot (1-p),
\end{align*}
by the law of total probability. We further have
\begin{align*}
\PPo{I_u=1, I_{u'}=1 \mid uu' \in E} 
&= 	\PPo{I_u=1\mid uu' \in E} \cdot \PPo{I_{u'}=1 \mid uu' \in E},\\
\PPo{I_u=1,I_{u'}=1 \mid uu' \not \in E} 
&= \PPo{I_u=1\mid uu' \not \in E} \cdot \PPo{I_{u'}=1 \mid uu' \not \in E},
\end{align*}
because
the events $\{I_u=1 \}$ and $\{ I_{u'}=1 \}$ depend only on the edges that are incident to $u$ and $u'$, respectively. This is the case, as we are working on the conditional space where the initial state of the vertices has been realised and the states of $u$ and $u'$ after the first round  depend only on the  edges that are incident to these two vertices. 
Thus, if we condition on the status of the pair $uu'$, that is, whether it is an edge or not, then 
the events $\{ I_u= 1 \}$ and $\{I_{u'}=1\}$ are independent. 
Thus, the first term of \eqref{eq:cov0} becomes 
\begin{align}
\PPo{I_u=1,I_{u'}=1} 
&=\PPo{I_u=1\mid uu' \in E} \cdot \PPo{I_{u'}=1 \mid uu' \in E}\cdot p\nonumber\\ 
&\hspace{2.5cm}+\PPo{I_u=1\mid uu' \not \in E} \cdot \PPo{I_{u'}=1 \mid uu' \not \in E}\cdot (1-p).\label{eq:firstterm}
\end{align}

Furthermore, by the law of total probability, the probabilities in the second term of \eqref{eq:cov0} can be written as
\begin{align*}
\PPo{I_{u}=1} 
&= \PPo{I_{u}=1 \mid uu' \in E} \cdot \PPo{uu'\in E} + 
\PPo{I_{u}=1 \mid uu' \not \in E} \cdot \PPo{uu' \not \in E}\\
&= \PPo{I_{u}=1 \mid uu' \in E} \cdot p + 
\PPo{I_{u}=1 \mid uu' \not \in E} \cdot (1-p),\\
\PPo{I_{u'}=1}&= \PPo{I_{u'} =1 \mid uu'\in E}\cdot  p
+ \PPo{I_{u'} =1 \mid uu'\not \in E }\cdot  (1-p).
\end{align*}
Thus, the second term of \eqref{eq:cov0} becomes
\begin{align}
  \PPo{I_{u}=1} \cdot \PPo{I_{u'}=1} 
  &= \Big(\PPo{I_{u}=1 \mid uu' \in E} \cdot p + 
  \PPo{I_{u}=1 \mid uu' \not \in E} \cdot (1-p)\Big)\times \nonumber\\ &\hspace{1.5cm}\Big(\PPo{I_{u'} =1 \mid uu'\in E}\cdot  p
  + \PPo{I_{u'} =1 \mid uu'\not \in E }\cdot  (1-p)\Big).\label{eq:secondterm}
\end{align}

To ease notation, letting 
\begin{align*}
A_u&:=\PPo{I_{u}=1 \mid uu' \in E}, \quad
A_{u'}:=\PPo{I_{u'}=1 \mid uu' \in E}\\
B_{u}&:=\PPo{I_u=1\mid uu' \not \in E},\quad
B_{u'}:=\PPo{I_{u'}=1 \mid uu' \not \in E}
\end{align*}
and plugging \eqref{eq:firstterm} and \eqref{eq:secondterm} into \eqref{eq:cov0}, we obtain
\begin{align*}
\Cov{I_u}{I_{u'}}
& = p A_{u}A_{u'} + (1-p) B_{u}B_{u'} - \Big(pA_{u} + (1-p) B_{u} \Big)\cdot\Big(p A_{u'}+ (1-p) B_{u'} \Big)\\
& = p (1-p)  (A_{u} -B_{u} ) (A_{u'}-B_{u'}),
\end{align*}
as claimed.
\end{proof}

Next we will estimate $|\PPo{I_u=1\mid uu' \in E}-\PPo{I_u=1\mid uu' \not \in E}|$. 
First observe that the event $\{I_u=1\}$ on either of the two 
conditional spaces (i.e., $\{uu' \in E\}$ or $\{uu' \not \in E\}$) is a function of the same collection of independent Bernoulli-distributed 
random variables, namely the indicators of $uu''\in E$, for any $u'' \not = u'$. However, 
the functions that determine $\{I_u=1\}$  that are associated with the conditional spaces differ 
only slightly. 

We shall rely on the following claim. 
\begin{claim} \label{clm:close-sums}
	Let $\{Y_i\}_{i\in I \cup I'}$ be a family of i.i.d. Bernoulli-distributed random variables, where 
	the index sets $I,I'$ are disjoint.  
	Then for any integer $a$ we have 
	\begin{equation*}
	\PPo {\sum_{i\in I}Y_i + a +1\geq \sum_{i\in I'} Y_i}  - 
	\PPo {\sum_{i\in I}Y_i + a \geq \sum_{i\in I'} Y_i} \leq \max_j \left\{ \PPo{\sum_{i\in I'} Y_i=j} \right\}.
	\end{equation*} 
\end{claim}
\begin{proof}[Proof of Claim~\ref{clm:close-sums}]
	Note that
	\begin{equation*}
	\begin{split}
	&\PPo {\sum_{i\in I}Y_i + a +1\geq \sum_{i\in I'} Y_i}  - 
	\PPo {\sum_{i\in I}Y_i + a \geq \sum_{i\in I'} Y_i}=\PPo {\sum_{i\in I}Y_i + a +1= \sum_{i\in I'} Y_i}.
	\end{split}
	\end{equation*}
	The result follows as 
	\begin{equation*}
	\PPo {\sum_{i\in I}Y_i + a +1= \sum_{i\in I'} Y_i}=\sum_{j=0}^{|I|+a+1}\PPo {\sum_{i\in I}Y_i + a +1= j}\PPo {\sum_{i\in I'} Y_i=j}.
	\end{equation*}
\end{proof}

We will apply the above claim in our setting in order to express the event $\{I_u =1 \}$. We set $I$ as the set of vertices in $V_n\setminus \{u,u',v\}$ with initial state $+1$, while $I'$ is the set of vertices in $V_n\setminus \{u,u',v\}$ with initial state $-1$, and for each $i \in I\cup I'$ the random variable $Y_i$ is the indicator 
that the corresponding edge exists. Setting $a=S_0(v)-{\bf 1}(S_0 (u) =-1)$ when $S_0(u')=+1$ and $a=S_0(v)-{\bf 1}(S_0 (u) =-1)+S_0(u')$ when $S_0(u')=-1$,
Claim~\ref{clm:close-sums} implies that 
\begin{equation*}
\left|\PPo{I_u=1 \mid uu'\in E}  -  \PPo{I_u =1 \mid uu'\not \in E }\right| \leq \max_j \left\{ \PPo{\sum_{i \in I'} Y_i = j} 
\right\}.
\end{equation*}
Now, note that $\sum_{i \in I'} Y_i$ follows the binomial distribution as a sum of $n/2-(1+o(1))c\sqrt{n}$ Bernoulli trials each having success probability $d/n$. 

Next we will distinguish between the cases $p\le 1-24 \gamma^2 n^{-1}$ and $1-24 \gamma^2 n^{-1}<p\le 1$, starting with the former.
The Local Limit Theorem (Theorem~\ref{cor:lll}) implies that 
$$  \max_j \left\{ \PPo{\sum_{i \in I'} Y_i = j} \right\} \le 2\sqrt{\frac{1}{2\pi (n/2-(1+o(1))c\sqrt{n}) p(1-p)}}\stackrel{c\le \sqrt{2\pi}/20}{\le} \sqrt{\frac{2}{np(1-p)}},$$ 
and thus
\begin{equation}
\left|\PPo{I_u=1 \mid uu'\in E}  -  \PPo{I_u =1 \mid uu'\not \in E }\right| \le \sqrt{\frac{2}{np(1-p)}}.\label{eq:cov1}
\end{equation}
An analogous argument implies
\begin{equation}
\left|\PPo{I_{u'}=1 \mid uu'\in E}  -  \PPo{I_{u'} =1 \mid uu'\not \in E }\right| \le \sqrt{\frac{2}{np(1-p)}}.\label{eq:cov2}
\end{equation}
Thus,~\eqref{eq:cov1} and \eqref{eq:cov2} in~\eqref{eq:covfin} yield
$$ \Cov{I_u}{I_{u'}} \le \frac{2}{n},$$
uniformly for all pairs $u,u' \in N(v)$. 
Since $N(v) < 2d$, for $n$ sufficiently large, and the variance of an indicator random variable is at most $1/4$ we then deduce that 
$$\Var{X_2 (v)} \le \frac{8 d^2}{n} + d \le 9d.$$
Now when $p>1-24\gamma^2 n^{-1}$, since the difference of any two probabilities is at most 1, we have by \eqref{eq:covfin} that
$$\Cov{I_u}{I_{u'}} \le p(1-p)\le \frac{24\gamma^2}{n},$$
and 
$$\Var{X_2 (v)}\le d+\frac{96\gamma^2 d^2}{n}\le (96\gamma^2+1)d.$$
This completes the proof of Lemma~\ref{lem:variance}.

\section{The last two rounds}\label{sec:end}
In the following lemma we show that if one starts the majority dynamics process from any configuration where the number of $-1$s is at most $\delta n$, for some $\delta$ small enough, then in two subsequent rounds unanimity will be achieved. 
\begin{lemma}\label{lem:lasttwo}
Let $d \geq \lambda n^{1/2}$ and $\delta < 1/10$.
Then with probability $1-o(1)$ the following holds for $G(n,p)$. For all partitions
 $P_0, N_0$ of  $V_n$, with $|P_0|\ge n(1- \delta)$, if all vertices in $P_0$ are in state +1, whereas 
all vertices of $N_0$ are in state $-1$, then after two rounds the majority dynamics process reaches unanimity. 
\end{lemma}
\begin{proof}[Proof of Lemma~\ref{lem:lasttwo}]
Let $P_i$ and $N_i$ denote the set of vertices in 
$+1$ and $-1$, respectively, after $i$ rounds.
Consider a partition of $V_n$ into two sets $P_0, N_0$ such that 
$|P_0|\geq n(1-\delta)$. Suppose that the majority dynamics starts with all elements of $P_0$ in state $+1$ 
and all elements of $N_0$ in state $-1$. Note that until this point we have only fixed the states of the vertices 
in the graph, but we have not exposed any edges so far. 

We will show that with probability $1- o(1)$ we have $|N_1| < d/10$. In order to achieve this, we bound the probability that every vertex in a set of size $d/10$ has state $-1$ after the first step and apply a union bound. 

For a subset of vertices $W$ we denote by $\{W \to N_1\}$ the event that 
after the first round all vertices in $W$ will have state $-1$.

We start by providing an upper bound on $\PPo{W \to N_1}$ for each $W\subset \VertexSet$ with $|W|=d/10$. 
For a vertex $v \in \VertexSet$ we let $d_{S}(v)$ denote its degree inside a subset of vertices $S$. 
This random variable is binomially distributed with parameters $|S|$ and $d/n$. Note that if $\{W \to N_1\}$, then for every $v\in W$ we have $d_{P_0}(v) \leq  d_{N_0}(v)$. Thus, we have the following upper bound:
$$\PPo{W\to N_1} \leq \PPo{\forall v \in W \ :\ d_{N_0} (v) \geq d_{P_0} (v)}.
$$
As $d_{N_0} (v) \le d_{N_0\setminus W} (v) + d_{W} (v) \leq 
d_{N_0\setminus W} (v) + |W|$ and $|W|=d/10$, we have that if 
$d_{N_0} (v) \geq d_{P_0} (v)$, then $d_{N_0\setminus W} (v) \geq d_{P_0}(v) - d/10 
\ge d_{P_0\setminus W}(v) - d/10$. Therefore
$$\PPo{W\to N_1} \leq \PPo{\forall v \in W \ :\ d_{N_0\setminus W} (v) \geq d_{P_0\setminus W}(v) - d/10}.
$$

The latter event is the intersection of independent events. 
For each one of them, we have 
\begin{equation*}
\begin{split}
\PPo{d_{N_0\setminus W} (v) \geq d_{P_0\setminus W}(v) - d/10} &< \PPo{  d_{P_0\setminus W} (v) < d/2} + 
\PPo{d_{N_0\setminus W}(v) > d/2 - d/10} \\
&\leq \PPo{  d_{P_0} (v) < d/2} + \PPo{d_{N_0\setminus W}(v) > d/3}.
\end{split}
\end{equation*}

Recall that $d_{S}(v)\sim \Bin{|S|, d/n}$ and $\delta <1/10$.
Thus $|P_0\setminus W|\ge n-\delta n -d/10\ge 8n/10$, whereby $d\ge \Ex{d_{P_0\setminus W}(v)}\ge 8d/10$. By the Chernoff bound, the first probability is $e^{-\Omega (d)}$.
On the other hand $|N_0\setminus W|\le |N_0|\le \delta n  < n/10$, as $\delta < 1/10$. 
Hence $\Ex{d_{N_0\setminus W}(v)}\le d/10$  and the Chernoff bound again implies that 
$\PPo{d_{N_0\setminus W}(v) > d/3} = e^{-\Omega (d)}$.  
So 
$$ \PPo{d_{N_0\setminus W} (v) \geq d_{P_0\setminus W}(v) - d/10}  = e^{-\Omega (d)},$$
whereby there exists $\lambda_1>0$ such that for $n$ sufficiently large, we have 
$$ \PPo{W\to N_1} \leq e^{-\lambda_1 d |W|}=e^{-\lambda_1 d^2/10}.$$
For such $n$, the union bound implies that the probability that there exists a set $W$ of size $d/10$ such that $\{W\to N_1\}$ holds is at most
$$\binom{n}{d/10}e^{-\lambda_1 d^2/10}\le \exp\left(\frac{d}{10}\left(\log{n}-\lambda_1 d\right)\right) < 
\exp \left( - \frac{\lambda_1 d^2}{20} \right) \leq \exp \left( - \frac{\lambda_1 \lambda^2 n}{20} \right).$$ 
Summing over all partitions of $V_n$ whose number  can be crudely bounded by $2^n$, the union bound 
implies that if $\lambda$ is sufficiently large, then the probability that there exists a subset $W$ is size $d/10$ 
which becomes negative after one step is $o(1)$. Note that this is the only part of the proof which uses the condition on $d$.

For the subsequent round, note that with probability $1-o(1)$, all vertices of $G(n,p)$ have degrees at least $d/2$. 
So if $|N_1|<d/10$, it turns out that after the execution of the first step 
all vertices will have the majority of their neighbours having state +1. Thus, the next round 
leads to unanimity. 
\end{proof}

\section{Reaching unanimity: proof of Theorem~\ref{thm:main}}\label{sec:proof-mainthm}
Let us fix $0<\eps < 1$. 
By Lemma~\ref{lem:initialstate}, with $c= \sqrt{2\pi}\eps /20$, we have 
$$\PPo{|\sum\nolimits_{v \in V_n}S_0 (v)| \geq 2c \sqrt{n}} > 1 - \eps /4, $$
provided that $n$ is sufficiently large. 
Conditional on this event, with probability $1/2$ we have $\sum_{v \in V_n}S_0 (v) \geq 2c\sqrt{n}$.
Let us assume that this event is realised. For the complementary case the proof is analogous.  

Let $P_2 :=\{ v : S_2 (v) = +1 \}$, that is, $P_2$ is the set of vertices whose state is $+1$ after the first 
two rounds. Let $N_2$ be the complement of this set.  
Lemma~\ref{lem:2-steps-lemma} implies that 
$$\Ex{|N_2| \mid \sum\nolimits_{v \in V_n}S_0 (v) \geq 2c\sqrt{n}} < \eps n /20. $$
So, by Markov's inequality, we have
$$\PPo{|N_2| < n/10 \mid \sum\nolimits_{v \in V_n}S_0 (v) \geq 2c\sqrt{n} } > 1 - \eps/2.$$
Finally, by Lemma~\ref{lem:lasttwo},  if $n$ is sufficiently large, then with probability at least $1-\eps/4$ 
the random graph $G(n,p)$ is such that after two more rounds unanimity will be reached. Thus, the union 
bound implies that with probability at least $1-\eps$  unanimity is reached after four rounds
and concludes the proof of Theorem~\ref{thm:main}.

\section{Local limit theorem: proof of Theorem~\ref{cor:lll}} \label{sec:proof_of_LLT}

We will use the following results in order to prove Theorem~\ref{cor:lll}.
Let $\mathbb{R}^+$ denote the set of positive real numbers.
\begin{theorem}[Theorem 6 in Chapter I of~\cite{MR0388499}]\label{thm:inv}
Let the random variable $X$ have lattice distribution, with possible values of the form $a+kh$ for some $a\in \mathbb R, h\in \mathbb R^+$, and any $k\in \mathbb Z$. Then, 
$$\PPo{X=a+kh}=\frac{h}{2\pi}\int_{|t|<\pi/h}\exp\left(-it(a+kh)\right)f(t) dt,$$
where $f(t)$ is the characteristic function of $X$, i.e., $f(t)=\Ex{\exp(itX)}$.
\end{theorem}

In particular, by taking $a=0, h=1$ in Theorem~\ref{thm:inv} we have for every integer-valued random variable $X$ and $k\in \mathbb{Z}$ that
\begin{equation}\label{eq:inv}
\PPo{X=k}=\frac{1}{2\pi}\int_{|t|<\pi}\exp\left(-itk\right)f(t) dt,
\end{equation}
where $f(t)$ is the characteristic function of $X$.

We also require a version of the Berry-Esseen Theorem.
\begin{lemma}[see eg. Lemma 1 in Chapter V  of~\cite{MR0388499}]\label{lem:BET}
Let $X_1,\ldots, X_n$ be independent random variables with $\Ex{|X_j-\Ex{X_j}|^3}< \infty$ for $j=1,\ldots,n$. In addition, let $X=\sum_{j=1}^n X_j$ and
$$L=(\Var{X})^{-3/2}\sum_{i=1}^n \Ex{|X_j-\Ex{X_j}|^3}.$$
Denote by $\hat{X}$ the normalised version of $X$, i.e., $\hat{X}=(X-\Ex{X})/\sqrt{\Var{X}}$, and by $\hat{f}(t)$ the characteristic function of $\hat{X}$, i.e., $\hat{f}(t)=\Ex{\exp(it\hat{X})}$.
Then  we have
$$|\hat{f}(t)-\exp(-t^2/2)|\le 16 L |t|^3 \exp(-t^2/3),$$
when $t\le 1/(4L)$.
\end{lemma}
Theorem~\ref{cor:lll} is a direct application of the following general local limit theorem for 
sum of Bernoulli-distributed random variables.
\begin{theorem}\label{cor:lll_general}
There exists an absolute constant $\gamma$ such that for any $n$ and set of independent Bernoulli-distributed random variables
$X_1,\ldots,X_n$ such that $X=\sum_{j=1}^n X_j$ satisfies $\Var{X}>0$
we have
$$\sup_{k\in \{0,\ldots, n\} }\left|\sqrt{\Var{X}}\PPo{X=k}-\frac{1}{\sqrt{2\pi}}\exp\left(-\frac{(k-\Ex{X})^2}{2\Var{X}}\right)\right|\le \frac{\gamma}{\sqrt{\Var{X}}}.$$
\end{theorem}

\begin{proof}

To ease notation, let  $\mu_j=\Ex{X_j}$,  $\sigma_j=\sqrt{\Var{X_j}}$ (for each  $j\in [n]$), $\mu=\Ex{X}$, and $\sigma=\sqrt{\Var{X}}$.
Throughout the proof we will often work with the normalised version of $X$, namely $\hat{X} = (X-\mu)/\sigma$. Note that $\PPo{X=k}=\PPo{\hat{X}=\hat{k}},$ where $\hat{k} =(k-\mu)/\sigma$. 
In addition, denote the characteristic functions of $X_j$ by $f_j(t)=\Ex{\exp(itX_j)}=\PPo{X_j=0}+\PPo{X_j=1}e^{it}$. Then the characteristic functions of $X$ and $\hat{X}$ satisfy
\begin{equation}\label{eq:charfunc}
f(t)=\prod_{j=1}^n f_j(t) \quad \mbox{and} \quad \hat{f}(t)=\exp\left(-\frac{it\mu}{\sigma}\right)f\left(\frac{t}{\sigma}\right). 
\end{equation}

By \eqref{eq:inv} we have for every $k\in \mathbb Z$ that,
\begin{equation}\label{eq:inversion}
\begin{split}
\sigma \cdot \PPo{X=k}&=\frac{\sigma}{2\pi}\int_{-\pi}^{\pi}\exp(-itk)f(t) dt\\
&=\frac{1}{2\pi}\int_{-\pi\sigma}^{\pi\sigma}\exp\left(-\frac{itk}{\sigma}\right)f\left(\frac{t}{\sigma}\right) dt\\
&\stackrel{\eqref{eq:charfunc}}{=}\frac{1}{2\pi}\int_{-\pi\sigma}^{\pi\sigma}\exp(-it\hat{k})\hat{f}(t) dt. 
\end{split}
\end{equation}
In addition, using 
$$\exp(-x^2/2)=\int_{-\infty}^{\infty}\exp(-ixt-t^2/2) dt,\quad \forall x\in \mathbb R,$$
we have
\begin{equation}\label{eq:complex}
\frac{1}{\sqrt{2\pi}}\exp\left(-\frac{(k-\mu)^2}{2\sigma^2}\right)=\frac{1}{\sqrt{2\pi}}\exp\left(-\hat{k}^2/2\right)=\int_{-\infty}^{\infty} \exp(-it\hat{k}-t^2/2)dt.
\end{equation}
By \eqref{eq:inversion} and \eqref{eq:complex}, we obtain that for every $k\in \mathbb Z$, 
\begin{equation}\label{eq:mainthm}
\begin{split}
&\left|\sigma \cdot \PPo{X=k}-\frac{1}{\sqrt{2\pi}}\exp\left(-\frac{(k-\mu)^2}{2\sigma^2}\right)\right|\\
&=\frac{1}{2\pi}\left|\int_{-\pi\sigma}^{\pi\sigma}\exp(-it\hat{k})\hat{f}(t) dt- \int_{-\infty}^{\infty} \exp(-it\hat{k}-t^2/2)dt\right|.
\end{split}
\end{equation}

To bound \eqref{eq:mainthm}, set 
$$L=\sigma^{-3}\sum_{i=1}^n \Ex{|X_j- \mu_j|^3}.$$
Since for each $j=1,\ldots,n$
\begin{equation*}
\begin{split}
\Ex{|X_j-\mu_j|^3}&=\PPo{X_j=0}\PPo{X_j=1}^3+\PPo{X_j=1}\PPo{X_j=0}^3\\
&=\PPo{X_j=0}\PPo{X_j=1}(\PPo{X_j=1}^2+\PPo{X_j=0}^2)\\
&=\sigma_j^2(\PPo{X_j=1}^2+\PPo{X_j=0}^2),
\end{split}
\end{equation*}
and because 
$$\frac{1}{4}\le\PPo{X_j=1}^2+\PPo{X_j=0}^2\le 1$$
we have
\begin{equation}\label{eq:thirdmoment}
\sigma_j^2/4\le \Ex{|X_j-\mu_j|^3}\le \sigma_j^2.
\end{equation}
In addition, $\sigma^2 =\sum_{i=1}^n \sigma_j^2$ (because $X_1,\ldots, X_n$ are independent),
thus 
\begin{equation}\label{eq:Lbound}
1/(4\sigma) \le L    \le 1/\sigma.
\end{equation}

Using $|\exp(-it\hat{k})|=1$ and $\eqref{eq:Lbound}$  we have 
\begin{equation}\label{eq:mainthm2}
\begin{split}
&\left|\int_{-\pi\sigma}^{\pi\sigma}\exp(-it\hat{k})\hat{f}(t) dt- \int_{-\infty}^{\infty} \exp(-it\hat{k}-t^2/2)dt\right|\\
&\le \int_{|t|\le 1/(4L)}\left|\hat{f}(t)- \exp(-t^2/2)\right|dt + \int_{1/(4L)\le |t| \le \pi \sigma}|\hat{f}(t)| dt + \int_{1/(4L)\le |t|} \exp(-t^2/2) dt.
\end{split}
\end{equation}

We will derive upper bounds for the three terms on the right-hand side  of \eqref {eq:mainthm2} one by one starting with the first term. 
Recall that $X=\sum_{j=1}^n X_j$ is the sum of Bernoulli random variables $X_j$ and by \eqref{eq:thirdmoment}  for every $j\in [n]$ we have 
$\Ex{|X_j- {\mu_j}|^3} 
 \le \sigma_j^2 \le  1/4$ (the last inequality is because $X_j$ is an indicator random variable).  In addition, $\hat{f}(t)$ is the characteristic function of the normalised random variable $\hat{X}$.
Therefore, by Lemma~\ref{lem:BET} we have 
\begin{equation}\label{eq:maineq1}
\int_{|t|\le 1/(4L)}\left|\hat{f}(t)- \exp(-t^2/2)\right|dt \le 16 L \int_{-\infty}^{+\infty}|t^3|\exp(-t^2/3) dt =144 L \stackrel{\eqref{eq:Lbound}}{\le} \frac{144}{\sigma}.
\end{equation}

Next we will consider the third term. Again using \eqref{eq:Lbound} we have
\begin{equation}\label{eq:maineq2}
\begin{split}
\int_{1/(4L)\le |t|} \exp(-t^2/2) dt
&\stackrel{\eqref{eq:Lbound}}{\le} \int_{\sigma/4\le |t|} \exp(-t^2/2) dt \\
&\le 2\int_{0}^{\infty} \exp(-(\sigma/4+t)^2/2) dt \\
&\le \exp\left(-\frac{\sigma^2}{32}\right)\int_{-\infty}^{\infty}  \exp(-t^2/2)  dt \le \frac{8}{\sigma}.
\end{split}
\end{equation}

Finally, we examine the second term. Using $\left|\exp\left(-\frac{it\mu}{\sigma}\right)\right| =1$  and by \eqref{eq:Lbound}, we obtain

\begin{equation*}
\begin{split}
\int_{1/(4L)\le |t| \le \pi \sigma}|\hat{f}(t)| dt
&=\int_{1/(4L)\le |t| \le \pi\sigma}\left|\exp\left(-\frac{it\mu}{\sigma}\right)f\left(\frac{t}{\sigma}\right)\right| dt\\
&=\int_{1/(4L)\le |t| \le \pi\sigma}\left| f\left(\frac{t}{\sigma}\right)\right| dt\\
&\stackrel{\eqref{eq:Lbound}}{\le} \sigma\int_{1/4\le |t| \le \pi}|f(t)| dt\\
&=\sigma \int_{1/4\le |t| \le \pi}\prod_{j=1}^{n}|f_j(t)| dt.
\end{split}
\end{equation*}

A quick calculation implies that for each $j=1,\ldots,n$
\begin{equation*}
\begin{split}
|f_j(t)|^2=\left|\Ex{e^{it X_j}}\right|^2&=|\PPo{X_j=0}+\PPo{X_j=1}e^{it}|^2\\
&=(\PPo{X_j=0}+\PPo{X_j=1}\cos(t))^2+(\PPo{X_j=1}\sin(t))^2\\
&=\PPo{X_j=0}^2+\PPo{X_j=1}^2+2\PPo{X_j=0}\PPo{X_j=1}\cos(t)\\
&=1+2\PPo{X_j=0}\PPo{X_j=1}(\cos{t}-1)\\
&=1+2 (\cos{t}-1)\sigma_j^2, 
\end{split}
\end{equation*}
where in the penultimate step we used $\PPo{X_j=0}^2+\PPo{X_j=1}^2=1-2\PPo{X_j=0}\PPo{X_j=1}$.

Together with $y^2\le e^{y^2-1}$ for every real $y$ and $\sigma^2 =\sum_{i=1}^n \sigma_j^2$ (because $X_1,\ldots, X_n$ are independent), this implies
$$\prod_{j=1}^{n}|f_j(t)|\le \exp\left(\frac{1}{2}(|f_j(t)|^2-1)\right) = \exp\left(\sum_{j=1}^n(\cos{t}-1)\sigma_j^2\right)=\exp\left((\cos{t}-1)\sigma^2\right).$$
Set $z=\cos(1/4)-1<0$ and note that for $1/4\le |t| \le \pi$ we have
$$\exp\left((\cos{t}-1)\sigma^2\right)\le \exp\left(z\sigma^2\right).$$
Since  $z<0$, we can bound the second term in \eqref{eq:mainthm2} from above by
\begin{equation}\label{eq:maineq3}
\int_{1/(4L)\le |t| \le \pi \sigma}|\hat{f}(t)| dt \le \sigma\int_{1/4\le |t| \le \pi}\prod_{j=1}^{n}|f_j(t)| \le 2\pi\sigma \exp\left(z\sigma^2\right)\le \frac{75}{\sigma}.
\end{equation}

Using the upper bounds~\eqref{eq:mainthm2}--\eqref{eq:maineq3} in \eqref{eq:mainthm} we obtain that for each $k\in \mathbb Z$, 
\begin{equation*}
\left|\sigma \cdot \PPo{X=k}-\frac{1}{\sqrt{2\pi}}\exp\left(-\frac{(k-\mu)^2}{2\sigma^2}\right)\right|\le  \frac{50}{\sigma},
\end{equation*}
completing the proof.
\end{proof}

\section{Discussion}\label{sec:disscussion}
In this paper we analyse the evolution of majority dynamics on $G(n,p)$ 
	with $p = d/n$ for   $d =d(n) \geq \lambda n^{1/2}$. Our main result is the proof of a conjecture of Benjamini 
et al.~\cite{ar:BenjChanmDonnelTamuz2016} 
in which majority dynamics on such a random graph will become unanimous in at most four steps with probability 
arbitrarily close to 1, provided that the initial state is selected uniformly at random and $\lambda$ and $n$ are sufficiently large.  

Of course, a natural question is how majority dynamics evolves on a random graph of {\em smaller} (average) degree $d$. Benjamini et al. made the following general conjecture.  
\begin{conjecture}[Benjamini, Chan, O'Donnel, Tamuz, Tan~\cite{ar:BenjChanmDonnelTamuz2016}] \label{conj:Benj2016}
With high probability over the choice of the random graph and the choice of the initial state the following holds.
\begin{enumerate} 
\item If $d \to \infty$, then for any $\eps>0$ and for any $n$ sufficiently large
$$\lim_{t \to \infty} \left|\sum\nolimits_{v \in V_n} S_{2t} (v)\right| \in [(1-\eps)n,n].$$
\item If $d$ is bounded, then for any $\eps>0$ and for any $n$ sufficiently large
$$\lim_{t \to \infty} \left| \sum\nolimits_{v \in V_n} S_{2t}(v) \right| \in [(1-\eps)n/2,(1+\eps)n/2].$$ 
\end{enumerate}
\end{conjecture}
In other words, when $d\to \infty$ (as $n\to \infty$), majority dynamics will be cyclic with period at most two, 
fluctuating between two states at which there is almost-unanimity. 
In this paper we verify this in a strong sense provided that $d\to \infty$ fast enough (Theorem~\ref{thm:main}). 
However, when $d$ is bounded, they conjectured that the process eventually reaches a cycle 
fluctuating between two states in which the vertices are approximately evenly split between the two states.

Strengthening the above, one can ask for the minimal difference required between the number of vertices with initial state $+1$ and $-1$ in order to eventually reach unanimity (with high probability). Our proof implies that for dense binomial random graphs this value is $O(\sqrt{n})$, however the exact threshold remains unknown.

Furthermore, a more detailed analysis is needed for $d$-regular $\lambda$-expanders where the initial state 
has bias of order $n^{-1/2}$ rather than $d^{-1/2}$, which is considered in~\cite{ar:MosselNeemanTamuz}. 
Benjamini et al.~\cite{ar:BenjChanmDonnelTamuz2016} proved that for random 4-regular graphs unanimity 
cannot be reached even if the bias is $\Omega (1)$. However, it is not clear whether for large $d$ (either 
fixed or moderately growing function of $n$) unanimity is reached even when the initial bias is of 
order $n^{-1/2}$.

\subsection*{Acknowledgement} We would like to thank anonymous referees for their helpful comments. 

\bibliographystyle{plain}
\bibliography{unanimity}
\end{document}